\documentclass{article} 

\usepackage[utf8]{inputenc}
\usepackage[T1]{fontenc}

\usepackage{amsmath, amssymb}
\usepackage{amsthm}

\usepackage{mathtools}

\usepackage{mathbbol}
\usepackage{amsfonts}
\DeclareSymbolFontAlphabet{\mathbb}{AMSb}
\DeclareSymbolFontAlphabet{\mathbbl}{bbold}
\usepackage{enumitem}
\usepackage{verbatim}
\usepackage{graphicx}

\usepackage{tikz-cd}

\usepackage{faktor} 

\usepackage[colorlinks,citecolor=blue,urlcolor=blue]{hyperref}

\newcommand{\Addresses}{{
  \bigskip
  \footnotesize

  F.~Gironella, \textsc{Rényi Institute of Mathematics, Budapest, Hungary}\par\nopagebreak
  \texttt{fabio.gironella@renyi.hu}

}}

\title{Examples of non-trivial contact mapping classes \\ for overtwisted contact manifolds in all dimensions}
\author{Fabio Gironella}
\date{}

\theoremstyle{plain}
\newtheorem{thm}{Theorem}
\numberwithin{thm}{section} 
\newtheorem{prop}[thm]{Proposition}
\newtheorem{lemma}[thm]{Lemma}

\newtheorem*{ex*}{Example}

\newtheorem*{remark*}{Remark}
\newtheorem*{falseexample*}{Naive reasoning}

\theoremstyle{definition}
\newtheorem{definition}[thm]{Definition}

\newcommand{\nat}{\mathbb{N}}
\newcommand{\integ}{\mathbb{Z}}
\newcommand{\rat}{\mathbb{Q}}
\newcommand{\real}{\mathbb{R}}
\newcommand{\compl}{\mathbb{C}}

\newcommand{\fieldk}{\mathbbl{k}}

\newcommand{\cercle}{\mathbb{S}^1}
\newcommand{\sphere}{\mathbb{S}^3}
\newcommand{\torus}{\mathbb{T}^2}
\newcommand{\realm}{\real^m}
\newcommand{\realmplusone}{\real^{m+1}}

\def\co{\colon\thinspace}

\newcommand{\CP}{\mathbb{CP}}

\DeclareMathOperator{\PD}{PD}

\DeclareMathOperator{\Id}{Id}

\DeclareMathOperator{\Hom}{Hom}

\DeclareMathOperator{\ator}{ator}

\newcommand{\homol}[2]{H_{#1}(#2;\integ)}
\newcommand{\homolM}[1]{\homol{#1}{M}}
\newcommand{\homolV}[1]{\homol{#1}{V}}

\newcommand{\cohomol}[2]{H^{#1}(#2;\integ)}

\newcommand{\cohomolV}[1]{\cohomol{#1}{V}}
\newcommand{\cohomolVator}{H^2_{\ator}(V;\integ)}
\newcommand{\cohomolVtorusator}{H^2_{\ator}(V\times\torus;\integ)}
\newcommand{\cohomolVSigmagator}{H^2_{\ator}(V\times\Sigma_g;\integ)}

\newcommand{\cohomolTwoV}{\cohomolV{2}}

\newcommand{\chern}{c}

\newcommand{\id}{\text{Id}}

\newcommand{\U}{\mathcal{U}}

\newcommand{\alphaplus}{\alpha_{+}}
\newcommand{\alphaminus}{\alpha_{-}}
\newcommand{\xiplus}{\xi_{+}}
\newcommand{\ximinus}{\xi_{-}}

\newcommand{\partialti}{\partial_{t_{i}}}
\newcommand{\partialthetazero}{\partial_{\theta_{0}}}
\newcommand{\partialthetai}{\partial_{\theta_{i}}}

\newcommand{\lineF}{\Phi}
\newcommand{\lineH}{H}
\newcommand{\lineX}{Y}
\newcommand{\lineFp}{\Psi}
\newcommand{\lineHp}{K}
\newcommand{\lineXp}{Z}

\newcommand{\diff}[1]{\mathcal{D}\left(#1\right)}
\newcommand{\diffV}{\diff{V}}
\newcommand{\diffVxi}{\diff{V,\xi}}
\newcommand{\ContStr}[1]{\Xi\left(#1\right)}
\newcommand{\ContStrV}{\ContStr{V}}

\DeclareMathOperator{\rang}{rk}

\begin{document}

	\maketitle

	\begin{abstract}
		We construct (infinitely many) examples in all dimensions of contactomorphisms of closed overtwisted contact manifolds that are smoothly isotopic but not contact-isotopic to the identity. 
	\end{abstract}	
	
	\section{Introduction}
	\label{SecIntro}
	
	One of the problems in the field of contact topology is to understand the topology of the space of contactomorphisms $\diffVxi$ of a given contact manifold $(V,\xi)$ in comparison with that of the space of diffeomorphisms $\diffV$ of the underlying smooth manifold $V$ or, more specifically, the problem of understanding the map $j_*:\pi_k\left(\diffVxi\right)\rightarrow\pi_k\left(\diffV\right)$ induced by the natural inclusion $j:\diffVxi\rightarrow\diffV$. 
	
	If $\ContStrV$ denotes the space of all the contact structures on $V$, in the case of closed manifolds the natural map $\diffV\rightarrow\ContStrV$ given by $\phi\mapsto \phi_*\xi$ helps to understand the properties of the $j_*$, and shows that the relation between the topology of $\diffVxi$ and that of $\diffV$ is mediated by the topology of $\ContStrV$. 
	Indeed, (the proof of) Gray's theorem implies, modulo a general fibration criterion, that this map is a locally-trivial fibration with fiber $\diffVxi$; see for instance Giroux and Massot \cite{MasGir15} for an explanation of this result or Massot \cite{MasFibrNotes} for a more detailed proof (the reader can also consult Geiges and Gonzalo Perez \cite{GeiGon04} for a proof of the fact that the map is a Serre fibration).
	Then, the exact long sequence of homotopy groups
	$$ \ldots \rightarrow \pi_{k+1}\left(\ContStrV\right)\rightarrow \pi_k\left(\diffVxi\right) \xrightarrow{j_*} \pi_k\left(\diffV\right)\rightarrow \pi_k\left(\ContStrV\right)\rightarrow\ldots $$	 
	associated to the fibration gives a relationship between the topologies of the three spaces $\diffV$, $\diffVxi$ and $\ContStrV$. \\
	
	As far as the $3$-dimensional case is concerned, the availability of classification results for the isotopy classes of tight contact structures on particular $3$-manifolds $V$ gives some explicit results about the lower homotopy groups in the long exact sequence above for these specific manifolds. The reader can consult Geiges and Gonzalo Perez \cite{GeiGon04}, Bourgeois \cite{Bou06}, Ding and Geiges \cite{DinGei10}, Geiges and Klukas \cite{GeiKlu14}, Giroux and Massot \cite{MasGir15} for results on $\pi_1\left(\ContStrV,\xi\right)$ as well as Giroux \cite{Gir01}, Giroux and Massot \cite{MasGir15} for results on $\pi_0 \left(\diffVxi\right)$. \\ 
	The situation in higher dimension is more complicated, due to the lack of classification results. The only results known so far are contained in Bourgeois \cite{Bou06}, Massot and Niederkr\"{u}ger \cite{MasNie16}, Lanzat and Zapolsky \cite{LanZap15}. In the first paper, Bourgeois gives results on some homotopy groups $\pi_k\left(\ContStrV,\xi\right)$, for particular contact manifolds $(V,\xi)$, using tools from contact homology. In \cite{MasNie16}, the authors give examples of contact manifolds $(V,\xi)$ for which $\ker\left(\pi_0\left(\diffVxi\right)\rightarrow\pi_0\left(\diffV\right)\right)$ is non-trivial; these examples rely on constructions in Massot, Niederkr\"{u}ger and Wendl \cite{MNW13}, which we will also use in the following. The last paper, dealing with the non-compact case, contains examples of embeddings of braid groups in the contactomorphism group of contactizations of certain non-compact symplectic manifolds.
	
	All the examples recalled so far are given on \emph{tight} contact manifolds. For the $3$-dimensional case, the dichotomy tight-overtwisted is well known since Eliashberg \cite{Eli89} and plays an important role in the classification results on which the cited examples are based.
	In the higher dimensional case, a clear definition of overtwistedness is given in Borman, Eliashberg and Murphy \cite{BorEliMur15} and according to it the three examples above are also tight.\\

	As far as the class of overtwisted manifolds is concerned, the only result known at the moment is the classification result of the path components of the space of contactomorphisms for all overtwisted contact structures on the $3$-sphere. 
	This result is attributed to Chekanov, according to Eliashberg and Fraser \cite[Remark 4.16]{EliFra09}; a written proof first appeared in the literature in Vogel \cite{Vog16}, where, among other things, the author proves, using $3$-dimensional techniques, that the space of embeddings of overtwisted disks in one of the overtwisted contact structures on $\sphere$ is not path-connected. 
	This gives in particular the first known examples of contactomorphisms of overtwisted $3$-manifolds that are smoothly isotopic but not contact-isotopic to the identity (we recall that, according to Cerf \cite{Cer68}, each orientation-preserving diffeomorphism of the $3$-sphere is smoothly isotopic to the identity).
	
	In this article we give other explicit examples of overtwisted $(V,\xi)$ such that the kernel of $\pi_0\left(\diffVxi\right)\rightarrow\pi_0\left(\diffV\right)$ is non-trivial. Though, we bypass here the problem of understanding the $\pi_0$ of the space of embeddings of overtwisted disks, about which nothing is known so far in high dimensions; the advantage of our approach is then that it gives (infinitely many) examples in each odd dimension.
	
	More precisely, we start by proving the following result:
	\begin{thm}
		\label{ThmExNonTrivContMappClass}
		Consider a closed manifold $W$ of dimension $2n\geq 2$ and let $\xi$ be a co-orientable contact structure on the manifold $V\coloneqq \cercle \times W$. 
		Suppose that the first Chern class $c_1(\xi)\in \cohomolV{2}$ 
		is \emph{toroidal} 
		and that, for each natural $k\geq 2$, the pullback $\pi_k^*\xi$ of $\xi$ via the $k$-fold cover $\pi_k \co \cercle\times W\rightarrow \cercle\times W$ given by $\pi_k(s,p)=(ks,p)$ satisfies 
		$c_1(\pi_k^*\xi)=k \cdot c_1(\xi)$ \emph{modulo} the submodule $\cohomolVator$ of \emph{atoroidal classes}.
		\\
		Then, the contact transformation $f:(\cercle\times W,\pi_k^*\xi)\rightarrow(\cercle\times W,\pi_k^*\xi)$ defined by $f(s,p) = (s+\frac{2\pi}{k},p)$ is smoothly isotopic but not contact-isotopic to the identity.
	\end{thm}
	Recall that a class $c\in\cohomolV{2}$ is called \emph{toroidal} if there is $f\co \torus\rightarrow V$ such that $f^*c \neq 0\in \cohomol{2}{\torus}$, and \emph{atoroidal} otherwise.
	\begin{remark*}
		Theorem \ref{ThmExNonTrivContMappClass} also holds (with similar proof) if one exchanges
		\begin{enumerate}[label=$(\ast)$]
			\item \label{EqHyp1} $c_1(\xi)$ is toroidal and, for each natural $k \geq 2$, $c_1(\pi_k^*\xi)=k \cdot c_1(\xi)\bmod\cohomolVator$,
		\end{enumerate}
		with the condition
		\begin{enumerate}[label=$(\ast')$]
			\item \label{EqHyp2} $c_1(\xi)$ is not torsion and, for each natural $k \geq 2$, $c_1(\pi_k^*\xi)=k \cdot c_1(\xi)$.
		\end{enumerate}
		Notice that $a\in\cohomolTwoV$ is toroidal if and only if $[a]\in
		\faktor{\cohomol{2}{V}}{\cohomolVator}$ is not torsion, because $\cohomol{2}{\torus}\simeq\integ$. In particular, \ref{EqHyp1} is equivalent to 
		$$\text{$c_1(\xi)$ is not torsion modulo $\cohomolVator$ and $c_1(\pi_k^*\xi)=k \cdot c_1(\xi)\bmod\cohomolVator$,}$$
		hence it is just a variation modulo $\cohomolVator$ of \ref{EqHyp2} (and it is not stronger nor weaker than \ref{EqHyp2}).
		\\
		Slightly anticipating what follows, we also point out that the contact structures given in Theorem \ref{ThmExplicitConstr}, Proposition \ref{PropExamplesHPrinciple} and Theorem \ref{ThmExamplesOBDBourgConstr}.\ref{Item1PropExamplesOBDBourgConstr} below actually satisfy both \ref{EqHyp1} and \ref{EqHyp2}; 
		on the other hand, working modulo $\cohomolVator$, i.e. with  \ref{EqHyp1}, is necessary for Theorem \ref{ThmExamplesOBDBourgConstr}.\ref{Item2PropExamplesOBDBourgConstr} .
		\\
		We hence decided to formulate everything in terms of \ref{EqHyp1}, even though \ref{EqHyp2} would give (everywhere but in Theorem \ref{ThmExamplesOBDBourgConstr}.\ref{Item2PropExamplesOBDBourgConstr}) slightly more direct proofs.
	\end{remark*}
	
	We then give, for each natural $n\geq 1$, an infinite number of \emph{explicit} overtwisted contact manifolds $(\cercle \times W^{2n},\xi)$ satisfying the hypothesis of Theorem \ref{ThmExNonTrivContMappClass}: 
	\begin{thm}
		\label{ThmExplicitConstr}
		Let $(M^{2n-1},\alphaplus,\alphaminus)$ be one of the infinitely many Liouville pairs constructed in Massot, Niederkr\"{u}ger and Wendl \cite{MNW13}. 
		Consider the (co-orientable) contact structure $\eta = \ker\left(\frac{1+\cos\left(s\right)}{2}\alphaplus + \frac{1-\cos\left(s\right)}{2}\alphaminus + \sin\left(s\right)dt\right)$ on the manifold $V\coloneqq \torus_{(s,t)}\times M$ (here, the notation $\torus_{(s,t)}$ denotes the choice of coordinates $(s,t)$ on $\torus$) and denote by $\xi$ the overtwisted contact structure obtained from $\eta$ via a half Lutz-Mori twist along $\{(0,0)\}\times M$, as defined in Massot, Niederkr\"{u}ger and Wendl \cite{MNW13}. 
		\\
		Then, $c_1(\xi)\in\cohomolV{2}$ is 
		toroidal 
		and, for each natural $k\geq2$, we have $c_1(\pi_k^*\xi)=k \cdot c_1(\xi) \bmod \cohomolVator$, where $\pi_k \co \torus_{(s,t)}\times M\rightarrow \torus_{(s,t)}\times M$ is given by $\pi_k(s,t,q)=(ks,t,q)$.
	\end{thm}
	
	\begin{ex*}
		If $n=3$, $(M,\alpha_{\pm})= (\cercle,\pm d\theta)$. Moreover, if $k=2$, the contact structure $\pi_2^*\xi$ on $V:=\torus\times M$ is the unique (up to isotopy) contact structure which is 
		invariant by the left-action by multiplication of $M=\cercle$ on $V$, invariant by the $f(s,t,\theta)=(s+\pi,t,\theta)$ defined in the statement and such that each torus $\torus_{(s,t)}\times \{\theta_0\}$ is convex with dividing set as in Figure \ref{FigDivSet}. 
		Theorems \ref{ThmExplicitConstr} and \ref{ThmExNonTrivContMappClass} then say that $f$ is not contact-isotopic to the identity;
		to our knowledge, even in this simple and very explicit setting, there is no trace of this result in the literature. 
	\end{ex*}
	
	\begin{figure}
		\centering{
			\def\svgwidth{150pt}
			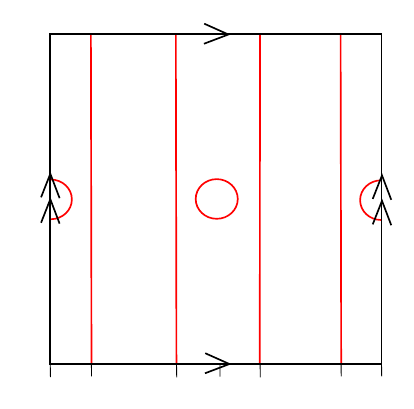
			\caption{Dividing set, in red, on the torus $\torus_{(s,t)}\times\{\theta_0\}$.}} 
		\label{FigDivSet}
	\end{figure}
	
	If one is just interested in giving examples of non$-$trivial elements in the kernel of the map $\pi_0\left(\diffVxi\right)\rightarrow\pi_0\left(\diffV\right)$ in each odd dimension, without wanting the underlying overtwisted contact manifolds $(V,\xi)$
	to be as explicit as those from Theorem \ref{ThmExplicitConstr}, the following result can also be proven using the existence of adapted open book decompositions proven by Giroux \cite{Gir01}:
	\begin{thm}
		\label{ThmExamplesOBDBourgConstr}
		Consider $W$ a closed $2n-$dimensional manifold and $\eta$ a co-orientable overtwisted contact structure on $V\coloneqq \cercle\times W$.
		Suppose that
		$c_1(\eta)$ is 
		toroidal  
		and that, for each $k\geq 2$, the pullback of $\eta$ via the $k$-fold covering $\pi_k \co V\rightarrow V$, given by $\pi_k(s,p)=(ks,p)$, satisfies $c_1(\pi_k^*\eta)=k \cdot c_1(\eta) \bmod \cohomolVator$.
		Then: 
		\begin{enumerate}[label=\roman*.]
			\item \label{Item1PropExamplesOBDBourgConstr} Each contact structure $\xi$ on $V\times\torus$ obtained via the Bourgeois construction \cite{Bou02} from $(V,\eta)$ (is co-orientable and) has first Chern class also satisfying the above conditions, with respect to the covering $\mu_k\coloneqq(\pi_k,\Id) \co V\times \torus\rightarrow V\times\torus$.
			\item \label{Item2PropExamplesOBDBourgConstr} Let $\nu\co V \times \Sigma_g \rightarrow V \times \torus$ be induced by a covering $\Sigma_g\rightarrow \torus$ branched over two points (here, $\Sigma_g$ denotes the closed surface of genus $g\geq 2$).
			Then, every contact branched covering $\xi_g$ of $\xi$ on $V \times \Sigma_g$ (is co-orientable and) has first Chern class satisfying the above conditions, with respect to the covering $\mu_k^g\coloneqq(\pi_k,\Id) \co V\times \Sigma_g\rightarrow V\times\Sigma_g$.
			Moreover, if $\eta$ is overtwisted and $g$ is large enough, $\xi_g$ is also overtwisted.
		\end{enumerate} 
	\end{thm}

	By an induction on the dimension, 
	Theorem \ref{ThmExamplesOBDBourgConstr} gives, for any integer $n\geq2$, examples of $(\cercle\times W^{2n},\xi)$ whose first Chern class satisfies the desired conditions.
	As far as point \ref{Item2PropExamplesOBDBourgConstr} is concerned, the reader can consult Geiges \cite{Gei97} for a construction and Gironella \cite{MyPaper17b} for a definition of \emph{contact branched coverings}.
	We also point out that the optimal integer $g$ to guarantee overtwistedness of $\eta_g$ 
	is actually $2$, according to an observation due to Massot and Niederkr\"uger (see Gironella \cite[Observation 5.10]{MyPaper17b}).
	
	Using the h-principle of Borman, Eliashberg and Murphy \cite{BorEliMur15}, an even bigger class of (non-explicit) examples can be obtained:
	\begin{prop}
		\label{PropExamplesHPrinciple}
		Consider a closed connected manifold $W^{2n}$ which is almost complex, spin and satisfies $\cohomol{1}{W}\neq \{0\}$.
		Then, there is a co-orientable overtwisted contact structure $\xi$ on $V\coloneqq \cercle\times W$ such that $c_1(\xi)$ 
		is toroidal  
		and $c_1(\pi_k^*\xi)=k \cdot c_1(\xi) \bmod \cohomolVator$, where $\pi_k \co \cercle_s\times W\rightarrow  \cercle_s\times W$ is given by $\pi_k(s,p)=(ks, p)$.
	\end{prop}

	\paragraph{Outline}
	Section \ref{ProofThmExNonTrivContMappClass} contains a proof by contradiction of Theorem \ref{ThmExNonTrivContMappClass}. Assuming that the contactomorphism $f$ is contact-isotopic to the identity, we construct a	contactomorphism between two contact structures $\xi_1$ and $\xi_2$; on the other hand, the hypothesis on the first Chern class of $\xi$ implies that $\xi_1$ and $\xi_2$ are not even isomorphic as almost contact structures.
	
	Section \ref{SecExFromLiouPairsHalfLutzMoriTwists} shows how to obtain examples of contact manifolds $(\cercle\times W^{2n},\xi)$ satisfying the hypothesis of Theorem \ref{ThmExNonTrivContMappClass} starting from Massot, Niederkr\"uger and Wendl \cite{MNW13}. \\
	More precisely, Section \ref{SubSecLutzMoriTwist} and \ref{SubSecConstrLiouvillePair} recall, respectively, the definition of half Lutz-Mori twist and the explicit constructions of Liouville pairs, both from Massot, Niederkr\"{u}ger and Wendl \cite{MNW13}.
	Then Section \ref{SubSecHalfLutzMoriTwistChernClasses} 
	describes the effects of a half Lutz-Mori twist on Chern classes in this context and Section \ref{SubSecProofPropExplicitConstr} contains a proof of Theorem \ref{ThmExplicitConstr}.
	
	Finally, in Section \ref{SecExFromHPrinciple} we show how to get examples of contactomorphisms smoothly isotopic but not contact-isotopic to the identity using the existence of adapted open book decompositions proven by Giroux \cite{Gir02} and the h-principle of Borman, Eliashberg and Murphy \cite{BorEliMur15}.
	More precisely, Theorem \ref{ThmExamplesOBDBourgConstr} and Proposition \ref{PropExamplesHPrinciple} are proven in Sections \ref{SubSecProofPropHPrinciple} and \ref{SubSecProofPropOBDBourgConstr} respectively.

	\section*{Acknowledgements}
	\phantomsection
	\addcontentsline{toc}{section}{Acknowledgements}	  
	
	This work is part of my PhD thesis \cite{MyPhDThesis}, written at Centre de mathématiques Laurent Schwartz of École polytechnique (Palaiseau, France).
	
	I thank my PhD advisor P Massot for introducing me to the problem, for sharing with me useful points of view and for the many comments that greatly improved this work.
	
	I also show my gratitude to A Oancea, who first asked whether the examples in Theorem \ref{ThmExplicitConstr}, originally only studied in dimension $3$, could be generalized to all dimensions, and to S Courte and E Murphy for very interesting questions, remarks and discussions about the rigid or flexible nature of the results in this manuscript. 
	
	Lastly, I would like to thank the anonymous referee for pointing out that
	the original main result (as it was stated in the previous version of this preprint, arXiv:1704.07278v2) on non-trivial contact mapping classes on manifolds obtained using constructions in Massot, Niederkr\"uger and Wendl \cite{MNW13} could actually be split as Theorems \ref{ThmExNonTrivContMappClass} and \ref{ThmExplicitConstr} above, and for suggesting that other examples satisfying Theorem  \ref{ThmExNonTrivContMappClass} could also be obtained as in Theorem \ref{ThmExamplesOBDBourgConstr} and Proposition \ref{PropExamplesHPrinciple}.

	\section{Proof of Theorem \ref{ThmExNonTrivContMappClass}} 
	\label{ProofThmExNonTrivContMappClass}	
	
	As each contactomorphism gives in particular an isomorphism of the underlying almost contact structures, Theorem \ref{ThmExNonTrivContMappClass} directly follows from the two following lemmas:
	\begin{lemma}
		\label{Lemma1ProofMainThm}
		Let $(\cercle \times W^{2n},\xi)$ be a contact manifold, with $\xi$ co-orientable.
		For each natural $k\geq2$, denote by $\pi_k\co \cercle \times W\rightarrow \cercle \times W $ the $k-$fold cover $\pi_k(s,p)=(ks,p)$ and by $f\co (\cercle \times W,\pi_k^*\xi)\rightarrow(\cercle \times W,\pi_k^*\xi)$ the contactomorphism $f(s,p)=(s+\frac{2\pi}{k},p)$. \\
		If $f$ is contact-isotopic to the identity, then there is a contactomorphism 
		$$\phi:(\cercle \times W,\pi_{kN}^*\xi)\overset{\sim}{\longrightarrow} (\cercle \times W,\pi_{kN+1}^*\xi) \text{ .}$$
	\end{lemma}
	\begin{lemma}
		\label{Lemma2ProofMainThm}
		Let $(V\coloneqq\cercle \times W,\xi)$, $\pi_k$ and $f$ be as in Lemma \ref{Lemma1ProofMainThm}.
		If moreover 
		$c_1(\xi)$ is toroidal 
		and $c_1(\pi_m^*\xi)=m\cdot c_1(\xi) \bmod \cohomolVator$ for every natural $m\geq2$, then $\pi_{m}^*\xi$ and $\pi_{m+1}^*\xi$ are not isomorphic as almost contact structures. 
	\end{lemma}
	
	We now prove Lemmas \ref{Lemma1ProofMainThm} and \ref{Lemma2ProofMainThm} above.
	
	\begin{proof}[Proof (Lemma \ref{Lemma1ProofMainThm})]
		In order to find the desired contactomorphism $\phi$, we use an idea  that already appeared in Geiges and Gonzalo Perez \cite{GeiGon04} and in Marinkovi\'c and Pabiniak \cite{MarPab16}, and which consists in cutting off contact hamiltonians on a particular cover of the manifold we are working with. 
		
		By hypothesis, the contactomorphism $f:(\cercle \times W,\pi_k^*\xi)\rightarrow(\cercle \times W,\pi_k^*\xi)$ defined by $f(s,p) = (s+\frac{2\pi}{k},p)$ is contact isotopic to the identity. Call $(F_r)_{r\in[0,1]}$ the isotopy, so that $F_0=\id$, $F_1 = f$ and $F_r$ is a contactomorphism for all $r\in[0,1]$.
		
		Take now the universal cover $\real_s$ of the factor $\cercle_s$ of the manifold $\cercle_s\times W$. Then, pull back $\pi_k^*\xi$ to a contact structure $\eta_k$ on the covering $\real_s\times W$ of $\cercle_s\times W$ and lift the contact isotopy $F_r$ to a contact isotopy $\lineF_r$ of $(\real_s\times W,\eta_k)$ starting at the identity. Fix a certain contact form $\beta_k$ for $\eta_k$ and denote by $\lineH_r\co \real_s\times W \rightarrow\real$ the path of contact hamiltonians $\beta_k(Y_r)$ associated to the contact vector field $Y_r$ generating the isotopy $\lineF_r$ (see for instance Geiges \cite[Section 2.3]{Gei08} for more details on contact hamiltonians).
		
		Now, by compactness of $W$ and $[0,1]$, there is an $N>0$ such that, for each $r\in[0,1]$, $\lineF_r(\{0\}_{s}\times W)$ is contained in $ \left(-2\left(N-1\right)\pi,+\infty\right)_s\times W$.
		
		Consider then an $\epsilon>0$ very small and a smooth function $\rho:\real\rightarrow\real$ such that $\rho(x)=0$ for $x<-2N\pi + \epsilon$ and $\rho(x)=1$ for $x>-2 \left(N-1\right) \pi-\epsilon$. We can then construct a new contact hamiltonian: $\lineHp_r(s,p):=\rho(s)\cdot \lineH_r(s,p)$, for all $(s,p)\in \real_s\times W$.
		
		We claim that the contact vector field $\lineXp_r$ associated to this new hamiltonian $\lineHp_r$ (i.e. the unique contact vector field $\lineXp_r$ such that $ \beta_k(\lineXp_{r})=\lineHp_r$; see for instance \cite[Section 2.3]{Gei08}) can be integrated to a contact isotopy $\left(\lineFp_r\right)_{r\in\left[0,1\right]}$ of $(\real_s\times W,\eta_k)$ starting at the identity. 
		Indeed, $\lineXp_r$ is zero for $s<-2N\pi + \epsilon$ and equal to the contact field $\lineX_r$ %associated to $\lineH_r$ 
		for $s>-2 \left(N-1\right) \pi-\epsilon$, which means in particular that it is integrable outside of a compact set of $\real_s\times W$ (remark that $\lineX_r$ is trivially integrable, because it comes from a contact isotopy); this implies integrability on all $\real\times W$. 
		Moreover, $\lineFp_r\vert_{\{0\}\times W}=\lineF_r\vert_{\{0\}\times W}$ and $\lineFp_r\vert_{\{-2N\pi\}\times W}=\Id\vert_{\{-2N\pi\}\times W}$ for all $r\in[0,1]$.
		
		In particular, $\lineFp_1$ maps $[-2 N\pi,0]\times W$ contactomorphically to $[-2 N \pi,\frac{2\pi}{k}]\times W$, where we consider on the domain and on the codomain the contact structure $\eta_k$. 
		
		Now, by the periodicity of $\eta_k$, we can identify the two boundary components of $[-2 N \pi,0]\times W$ so that the restriction of $\eta_k$ induces a well defined contact structure on the quotient. 
		More precisely, the quotient contact manifold obtained is $(\cercle_s\times W,\pi_{kN}^*\xi)$. \\ 
		The analogous procedure for the codomain $[-2 N \pi,\frac{2\pi}{k}]\times W$ of $\lineFp_1$ gives as quotient the contact manifold $(\cercle_s\times W,\pi_{kN+1}^*\xi)$.
		
		Lastly, because $\lineFp_1: [-2 N \pi,0]\times W \rightarrow [-2 N \pi,\frac{2\pi}{k}]\times W$ is the identity on a neighborhood of $\{-2N\pi\}\times W$ and a lift of the translation $f$ on a neighborhood of $\{0\}\times W$, it induces on the quotient contact manifolds a well defined contactomorphism
		\begin{equation*}
		\phi:(\cercle_s\times W,\pi_{kN}^*\xi)\overset{\sim}{\longrightarrow} (\cercle_s\times W,\pi_{kN+1}^*\xi) \text{ .} \qedhere
		\end{equation*}
	\end{proof}
	
	\begin{proof}[Proof (Lemma \ref{Lemma2ProofMainThm})]
		Suppose by contradiction that there is an isomorphism of almost contact structures $\psi: \left(V, \pi_{m}^*\xi\right) \overset{\sim}{\rightarrow} \left(V, \pi_{m+1}^*\xi\right)$;
		in particular, 
		\begin{equation}
		\label{Eq0Lemma2ProofMainThm}
		\psi_*\chern_1(\pi_{m}^*\xi)=\chern_1(\pi_{m+1}^*\xi) \text{ .}
		\end{equation}
		Because the submodule $\cohomolVator$ of atoroidal classes is natural (i.e. it is preserved by pullbacks induced by continuous maps $V\rightarrow V$), 
		$\psi_*$ induces a well defined endomorphism, which is moreover an isomorphism, of the $\integ$-module $N\coloneqq\faktor{\cohomolTwoV}{\cohomolVator}$.
		We then have $\psi_*(\pi_{n}^*\xi)=n \psi_*\chern_1(\xi) \bmod \cohomolVator$ for each natural $n\geq 2$, so that
		Equation \ref{Eq0Lemma2ProofMainThm} becomes 
		\begin{equation}
		\label{Eq1Lemma2ProofMainThm}
		m \psi_* c_1(\xi) = (m+1)c_1(\xi) \bmod \cohomolVator \text{ .}
		\end{equation}
		Notice also that $N$ is a finitely generated $\integ$-module without torsion.
		In particular, there is a well defined \emph{divisibility} map 
		\begin{align*}
		d \co N\setminus \{0\} &\rightarrow \nat\setminus\{0\} \\
		a & \mapsto \max\{\,k\in\nat\, \vert\, \exists b\in N,\; a=k b\, \}
		\end{align*}
		which also satisfies $d(ha)=hd(a)$ and $d(\psi_*a)=d(a)$, for each $a\in N\setminus \{0\}$ and $h\in\nat\setminus\{0\}$.
		Because $c_1(\xi)$ is toroidal, we can then apply $d$ to both the left and right hand sides of Equation \ref{Eq1Lemma2ProofMainThm}, thus obtaining the desired contradiction. 
	\end{proof}

	\section{Examples from Liouville pairs and half Lutz-Mori twists}
	\label{SecExFromLiouPairsHalfLutzMoriTwists}
	
	The idea of the proof of Theorem \ref{ThmExplicitConstr} is the following. 
	The contact structure $\eta$ on the manifold $V=\cercle\times W$ in the statement has trivial Chern classes (better, it is trivializable as complex bundle). 
	We then apply a semi-local modification to $\eta$ and obtain another contact structure $\xi$;
	the explicit nature of this modification (as well as the explicit nature of the original contact manifold $(V,\eta)$) allows us to compute the first Chern class of $\xi$, and to show that it satisfies the desired conditions. 
	
	This section is structured in the following way.
	We recall in Sections \ref{SubSecLutzMoriTwist} and \ref{SubSecConstrLiouvillePair}, respectively, the notion of half Lutz-Mori twist and the construction of Liouville pairs, both from Massot, Niederkr\"uger and Wendl \cite{MNW13}.
	We then describe in Section \ref{SubSecHalfLutzMoriTwistChernClasses} how half Lutz-Mori twists (along contact submanifolds belonging to one of the Liouville pairs constructed in \cite{MNW13}) affect the Chern classes of the underlying almost contact structure.
	Finally, Section \ref{SubSecProofPropExplicitConstr} contains the proof of Theorem \ref{ThmExplicitConstr}.
	
	\subsection{The half Lutz-Mori twist}
	\label{SubSecLutzMoriTwist}
	
	Developing some ideas introduced by Mori in \cite{Mor09} in the 5-dimensional case,  Massot, Niederkr\"{u}ger and Wendl introduce in \cite{MNW13} the notion of \emph{Lutz-Mori twist} along a submanifold belonging to a \emph{Liouville pair} as a generalization of the known $3$-dimensional Lutz twists. 
	In this section, we briefly recall how to perform the \emph{half} version of the Lutz-Mori twist, which we will use in the following. 
	
	We start by recalling the notion of Liouville pair:
	\begin{definition}{\cite{MNW13}}
		Let $M^{2n-1}$ be an oriented manifold. A \emph{Liouville pair on $M$} is a couple of contact forms $(\alphaplus,\alphaminus)$ such that $\pm \alpha_{\pm}\wedge\left(d\alpha_{\pm}\right)^{n-1}>0$ and such that the form $e^{r}\alphaplus+ e^{-r}\alphaminus$ is a Liouville form (i.e. its differential is symplectic) on $\real_{r}\times M$. 
	\end{definition}
	
	We point out that the existence of Liouville pairs on closed manifolds is not trivial; at the moment, the only known examples in high dimension are given by the construction in \cite[Section $8$]{MNW13}, which is nonetheless a source of infinitely many non-homeomorphic manifolds with Liouville pairs in each (odd) dimension. 
	In Section \ref{SubSecConstrLiouvillePair} we will recall the properties of this construction which are needed in order to prove Theorem \ref{ThmExNonTrivContMappClass}.\\
	
	Let now $(V,\eta)$ be a contact manifold having as a codimension-2 contact submanifold $(M,\xiplus)$ such that $\alphaplus$ defining $\xiplus$ belongs to a Liouville pair $(\alphaplus,\alphaminus)$. 
	We want to describe how to perform a half Lutz-Mori twist on $(V,\eta)$ along $(M,\xiplus)$.
	
	Consider then the $1$-form $\alpha = \frac{1+\cos\left(s\right)}{2}\,\alphaplus + \frac{1-\cos\left(s\right)}{2}\,\alphaminus + \sin\left(s\right)dt$ on $[\pi,2\pi]_{s}\times\cercle_t\times M$; notice that this is a contact form because $(\alphaplus,\alphaminus)$ is a Liouville pair on $M$.
	Let then $(U,\xi_U)$ be the \emph{blow-down} of $\left([\pi,2\pi]_{s}\times\cercle_t\times M,\ker\alpha\right)$ along $\{\pi\}\times\cercle_t\times M$, as defined in \cite[Section 5.1]{MNW13}.\\
	More explicitly, $(U,\xi_U)$ is obtained as follows.  
	\sloppy{The hypersurface (or, better, \emph{round hypersurface}, as defined in \cite[Section 5.1]{MNW13}) $\{\pi\}\times\cercle_t\times M$ admits a neighborhood of the form $\left([0,\epsilon)_x\times\cercle_t\times M, \ker \left(\alpha_{-}+x dt\right)\right)$ inside $\left([\pi,2\pi]_{s}\times\cercle_t\times M,\ker\alpha\right)$, in such a way that $\{\pi\}_s\times\cercle_t\times M$ corresponds to $\{0\}_x\times\cercle_t\times M$; this follows from the fact that the restriction of the two contact structures to the two hypersurfaces coincide (see \cite[Lemma 5.1]{MNW13}).
	We can then remove the hypersurface %$\{0\}_x\times\cercle_t\times M \simeq 
	$\{\pi\}_s\times\cercle_t\times M$ inside 
	%neighborhood $\left([0,\frac{1}{2}\epsilon)_r\times\cercle_t\times M, \ker \left(\alpha_{-}+r dt\right) \right)$ of $\{\pi\}\times\cercle_t\times M$ inside
	$\left([\pi,2\pi]_{s}\times\cercle_t\times M,\ker\alpha\right)$ and glue $\left(D^2_{\sqrt{\epsilon}}\times M, \ker \left(\alpha_{-}+r^2 d\varphi\right)\right)$ (here $(r,\varphi)$ are polar coordinates on the $2$-disc $D^2_{\sqrt{\epsilon}}$ centered at the origin and of radius $\sqrt{\epsilon}$) thanks to the contactomorphism from  $\left(\left(D^2_{\sqrt{\epsilon}}\setminus\{0\}\right)\times M, \ker \left(\alpha_{-}+r^2 d\varphi\right)\right)$ to  $\left((0,\epsilon)_x\times\cercle_t\times M, \ker \left(\alpha_{-}+x dt\right)\right)$ (seen as a subset of $\left((\pi,2\pi]_{s}\times\cercle_t\times M,\ker\alpha\right)$) given by $(r,\varphi,p) \mapsto (r^2,\varphi,p)$.
	The resulting contact manifold (with one boundary component) is the desired $(U,\xi_U)$.}

	At this point, performing a half Lutz-Mori twist along $(M,\xiplus)$ means replacing a neighborhood of $(M,\xiplus)$ in $(V,\eta)$ with $(U,\xi_U)$.\\
	More precisely, one can see that the boundary component $\{2\pi\}\times\cercle_t\times M$ of $(U,\xi_U)$ also admits a neighborhood $\left((-\epsilon,0]_x\times\cercle_t\times M, \ker \left(\alpha_{+}+x dt\right)\right)$ inside $(U,\xi_U)$, in such a way that $\{2\pi\}_s\times\cercle_t\times M$ corresponds to $\{0\}_x\times\cercle_t\times M$.
	Now, $(M,\xi_+)$ is a codimension $2$ contact submanifold with trivial normal bundle in $(V,\eta)$; hence, by the contact neighborhood theorem (see Geiges \cite[Theorem 2.5.15]{Gei08}), there is $\delta>0$ such that $(M,\xiplus)$ admits a neighborhood $\left(D^2_\delta \times M,\eta_0\coloneqq\ker\left(\alpha_+ + r^2 d\varphi\right)\right)$ inside $(V,\eta)$ (here, $(r,\varphi)$ are polar coordinates on $D^2_\delta$), in such a way that $(M,\xi_+)$ corresponds to $(\{0\}\times M,\eta_0\vert_{\{0\}\times M})$.
	Because \sloppy{$\left(\left(D^2_\delta\setminus\{0\}\right) \times M, \ker\left(\alpha_+ + r^2 d\varphi\right)\right)$} is contactomorphic to $\left(\left(0,\delta^2\right)_x\times\cercle_t \times M,\ker\left(\alpha_+ + x dt\right)\right)$ via $(r,\varphi,p)\mapsto(r^2,\varphi,p)$, we can then glue $(U,\xi_U)$ to $(V\setminus M,\eta)$ and obtain a well defined contact manifold $(V,\xi)$ (notice that the underlying smooth manifold is still $V$).
	
	The above construction does not depend, up to isotopy, on any choice made.
	
	\begin{definition}{\cite[Remark 9.6]{MNW13}} 
		\label{DefHalfLutzMoriTwist}
		$(V,\xi)$ is said to be obtained from $(V,\eta)$ by a \emph{half Lutz-Mori twist} along the contact submanifold $\left(M,\xiplus=\ker\left(\alphaplus\right)\right)$ belonging to the Liouville pair $(\alphaplus,\alphaminus)$.
	\end{definition}	 
	
	We point out that performing a half Lutz-Mori twist makes the contact manifold overtwisted. 
	Indeed, it is explained in Massot, Niederkr\"{u}ger and Wendl \cite[Remark 9.6]{MNW13} that this twist always gives a PS-overtwisted manifold, which then is also overtwisted according to Casals, Murphy and Presas \cite{CMP15} and Huang \cite{Hua16}.

	\subsection{Construction of Liouville pairs}
	\label{SubSecConstrLiouvillePair}
	
	We recall here the construction in Massot, Niederkr\"{u}ger and Wendl \cite[Section 8]{MNW13}, leaving the details that are not important for our purposes.
	
	Consider the product manifold $\realm\times\realmplusone$ with the pair of contact structures $\xiplus$, $\ximinus$ induced by the following pair of contact forms:
	$$ \alpha_{\pm} \, := \, \pm e^{t_1 + \ldots + t_m} d\theta_0 + e^{-t_1} d\theta_1 + \ldots + e^{-t_m} d\theta_m \text{ ,}$$ 
	where we use coordinates $(t_1,\ldots,t_m)$ on $\realm$ and $(\theta_0,\ldots,\theta_m)$ on $\realmplusone$. A direct computation shows that $(\alphaplus,\alphaminus)$ is a Liouville pair on $\realm\times\realmplusone$.
	
	We now remark that there are two Lie groups acting explicitly on $\realm\times\realmplusone$ by strict contact transformations for both $\alphaplus$ and $\alphaminus$.\\ 
	Indeed, the left action of the group $\realmplusone$ on $\realm\times\realmplusone$ given by the translations 
	\begin{equation*}
	\begin{multlined}
	(\varphi_0,\ldots,\varphi_m)\cdot(t_1,\ldots,t_m,\theta_0,\ldots,\theta_m)  := (t_1,\ldots,t_m ,\theta_0+\varphi_0,\ldots,\theta_m + \varphi_m)
	\end{multlined}	
	\end{equation*}		
	and the left action of $\realm$ given by the law 
	\begin{equation*}	
	\begin{multlined}
	(\tau_1,\ldots,\tau_m)\cdot(t_1,\ldots,t_m,\theta_0,\ldots,\theta_m) \, {} \\ {} := \, (t_1 + \tau_1,\ldots,t_m + \tau_m ,e^{-\tau_1 + \ldots -\tau_m}\theta_0, e^{\tau_1}\theta_1,\ldots,e^{\tau_m}\theta_m )
	\end{multlined}
	\end{equation*}		
	are Lie group left-actions on $\realm\times\realmplusone$ and they both preserve the contact forms $\alphaplus$ and $\alphaminus$.
	
	Moreover, these two actions allow us to produce a \emph{compact} contact manifold from $\realm\times\realmplusone$.
	Indeed, there are lattices $\Lambda,\Lambda'$ of $\realm$ and $\realmplusone$ respectively, such that the $\Lambda$-action on $\realm\times\realmplusone$ induced by the action of $\realm$ preserves $\realm\times\Lambda'$. 
	This implies that, by first taking the quotient of $\realm\times\realmplusone$ by $\Lambda'$ and then quotienting it by the (well defined by the above property) induced action of $\Lambda$, we obtain a compact manifold $M$. 
	
	Finally, this manifold $M$ naturally inherits a Liouville pair, still denoted by $(\alphaplus,\alphaminus)$, from the Liouville pair on the covering $\realm\times\realmplusone$, because $\realm$ and $\realmplusone$ act on $\realm\times\realmplusone$ by strict contactomorphisms for both $\alphaplus$ and $\alphaminus$.
	
	We point out that this construction actually gives an infinite number of non homeomorphic manifolds $M$, hence an infinite number of non isomorphic Liouville pairs, in each odd dimension bigger or equal to $3$. \\
	Indeed, the existence of the lattices $\Lambda$ and $\Lambda'$ follows from number theory arguments and the manifold $M$ obtained depends on the choice of a totally real field of real numbers $\fieldk$ with finite dimension over $\rat$. 
	Now, for each dimension $\geq 2$ over $\rat$, there are infinitely such fields $\fieldk$ and the corresponding manifolds are non homeomorphic. 
	See \cite[Lemma 8.3]{MNW13} for the details.\\
	As far as Theorem \ref{ThmExplicitConstr} is concerned, this means that we have, in each odd dimension $2n+1\geq5$, a contact structure satisfying the hypothesis of Theorem \ref{ThmExNonTrivContMappClass} on infinitely many different smooth manifolds $\torus\times M^{2n-1}$; in dimension $3$, we obtain one contact structure on $\torus\times M^1=\mathbb{T}^3$.
	In both cases, Theorem \ref{ThmExNonTrivContMappClass} then gives examples of contactomorphisms smoothly isotopic but not contact isotopic to the identity for the countably many contact structures $\left(\pi_k^*\xi\right)_{k\geq2}$ on each $\torus\times M$.
	
	\subsection{Effects of half Lutz-Mori twists on Chern classes}
	\label{SubSecHalfLutzMoriTwistChernClasses}
	
	Chern classes are global invariants of complex vector bundles $E$ over a manifold $V$. 
	In our setting, we then have to find a way to study how local modifications (i.e. over an open set $\U$ of $V$) of the complex vector bundle $E$ affect its Chern classes. 
	The solution is either to use a relative version of Chern classes or to shift to another point of view more local in nature.
	
	Aguilar, Cisneros-Molina and Fr\'\i as-Armenta \cite{ACF07} adopt in particular this second strategy and this allows them to prove a generalization of the classical fact that the top Chern class of $E$ is the Poincaré dual of the zero locus of a section of $E$ which is transverse to the zero section. 
	In order to achieve such generalization, they deal with the following technical issue:  when $1<k\leq r=\rang_\compl(E)$, the locus $S_k$ of points where $k$-sections $s_1,\ldots,s_k$ are $\compl$-linearly dependent may not be a smooth submanifold of $V$, even for a ``generic'' choice of $s_1,\ldots,s_k$, hence it has a priori no well defined homology class. 
	In \cite{ACF07} it is hence proved that $S_k$ can be desingularized to a smooth submanifold $Z_k$ of $V\times\CP^{k-1}$ in such a way that the $(r-k+1)$-th Chern class of $E$ can be interpreted as the Poincaré dual of the pushforward in $V$ of the class of $Z_k\subset V\times\CP^{k-1}$ via the map induced in homology by the projection $V\times\CP^{k-1}\rightarrow V$.

	In our context of half Lutz-Mori twists along particular contact submanifolds, the results proven by Aguilar, Cisneros-Molina and Fr\'\i as-Armenta \cite{ACF07} give the following:
	\begin{prop}
		\label{PropCompChernClass}
		Let $(V^{2m+3},\xi)$ be a contact manifold containing the $(M^{2m+1},\xiplus)$ of Section \ref{SubSecConstrLiouvillePair} as a codimension $2$ contact submanifold with trivial normal bundle. Then, if we denote by $\xi'$ the contact structure on $V$ obtained by performing a half Lutz-Mori twist along the submanifold $(M,\xiplus)$ (where we consider $M$ with the orientation given by $\xiplus$), we have the following:
		\begin{enumerate}
			\item \label{Item1PropCompChenClass} for all $i=2,\ldots, m+1$, $\chern_i(\xi') - \chern_i(\xi) = 0$ in $\cohomolV{2i}$;
			\item \label{Item2PropCompChenClass} $\chern_1(\xi') - \chern_1(\xi) = - 2 \PD\left(j_\ast \left[M\right]\right)$ in $\cohomolTwoV$, where $ j: M \rightarrow V $ is the inclusion, $j_{\ast}: \homolM{2m+1}\rightarrow\homolV{2m+1}$ is the induced map and $\PD(\alpha)$ denotes the Poincaré dual of the homology class $\alpha\in\homolV{\ast}$.
		\end{enumerate}
	\end{prop}
	
	\begin{remark*}
		This result is not in contradiction with Massot, Niederkr\"{u}ger and Wendl \cite[Theorem 9.5]{MNW13}, where the authors prove that the contact structures before and after a full Lutz-Mori twist (as defined in \cite[Section 9.1]{MNW13}) are homotopic through almost contact structures, hence have the same Chern classes.\\
		Indeed, the result $\xi''$ of a full Lutz-Mori twist can be interpreted as a couple of successive half twists. More precisely, we first perform a half twist along a submanifold $(M,\xiplus)$ to obtain $\xi'$; this changes the core of the tube where we perform the twist from $(M,\xiplus)$ to $(M,\ximinus)$. We then perform another half twist, this time along the new core $(M,\ximinus)$, to obtain $\xi''$. Hence, applying Proposition \ref{PropCompChernClass} twice and using the fact that $\ximinus$ induces an orientation that is opposite to that induced by $\xiplus$, we get that $\chern_i(\xi'')=\chern_i(\xi')=\chern_i(\xi)$ for all $i=2,\ldots, m+1$ and that $\chern_1(\xi'') = \chern_1(\xi') - 2 \PD\left(j_\ast \left[-M\right]\right) = \chern_1(\xi) - 2 \PD\left(j_\ast \left[M\right]\right)- 2 \PD\left(j_\ast \left[-M\right]\right) = \chern_1(\xi)$, as we expected from \cite[Theorem 9.5]{MNW13}.
	\end{remark*}

	The proof of 
	Proposition \ref{PropCompChernClass} relies on the 
	explicit results in \cite{ACF07};
	we hence made the choice to omit it in this paper, in order to avoid lengthy technical digressions and keep the focus on the  motivating contact geometric problem, i.e. the research of examples of contactomorphisms smoothly isotopic but not contact isotopic to the identity on overtwisted contact manifolds of high dimensions. 
	A detailed proof of Proposition \ref{PropCompChernClass} (together with the necessary background from \cite{ACF07}) can be found in Gironella \cite[Section 4.2.3 and Appendix A]{MyPhDThesis}.

	\subsection{Proof of Theorem \ref{ThmExplicitConstr}}
	\label{SubSecProofPropExplicitConstr}
	
	We use in this section the notations introduced in the statement of Theorem \ref{ThmExplicitConstr}. 
	
	The contact structure $\eta$ on the manifold $\torus_{(s,t)}\times M$ can be explicitly written as the kernel of $\alpha\coloneqq \sum_{i=1}^{m}e^{-t_i}d\theta_i+\cos(s) \, e^{\sum_{i=1}^{m}t_i}d\theta_0 + \sin(s)dt$, where we use locally on $M$ the coordinates $(t_1,\ldots,t_m,\theta_0,\ldots,\theta_m)$  induced by the covering $\real^m\times\real^{m+1}\rightarrow M$, as described in Section \ref{SubSecConstrLiouvillePair}.
	Then, $\eta$ admits a trivialization as a complex vector bundle given by the following sections and choice of $d\alpha\vert_\eta$-compatible complex structure $J$:
	\begin{enumerate}
		\item $S_i:=\partialti$ for $i=1,\ldots, m$, and $S_{m+1}\coloneqq \partial_s$
		\item \sloppy{$J(S_i):= e^{-\sum_{j=1}^{m}t_j}\cos\left(s\right)\partialthetazero - e^{t_i}\partialthetai + \sin\left(s\right) \partial_t$, for $i=1,\ldots,m$, and ${J(S_{m+1})\coloneqq -e^{-\sum_{j=1}^{m}t_j}\sin(s)\partial_{\theta_{0}} + \cos(s)\partial_t}$.}
	\end{enumerate}
	(An explicit computation shows that these sections are indeed well defined on $\torus_{(s,t)}\times M$ and not only on $\torus_{(s,t)}\times \real^{m}\times\real^{m+1}$).\\
	In particular, all the Chern classes of $\eta$ are zero.
	Hence, applying Proposition \ref{PropCompChernClass} to the couple $(\xi,\eta)$ we get the following: if we denote by $ j: M \rightarrow \torus_{(s,t)}\times M$ the inclusion $j(p) = (0,0,p)$ and by $j_{\ast}: \homolM{2m+1}\rightarrow\homol{2m+1}{\torus\times M\,}$ the induced map in homology, then $\chern_1(\xi) = - 2 \PD\left(j_\ast \left[M\right]\right)$ in $\cohomol{2}{\torus\times M\,}$.
	
	We now prove that $c_1(\xi)$ is toroidal.
	Fix a $p\in M$ and consider $f\co \torus\rightarrow \torus\times M$ given by $f(\theta,\varphi)=(\theta,\varphi,p)$, for every $(\theta,\varphi)\in \torus$.
	Because $f$ is transverse to $j(M)$, we have $f^*\PD_{\torus\times M} \left(j_\ast \left[M\right]\right) = \PD_{\torus}\left(\left[f^{-1}(j(M))\right]\right)$; 
	here, the notation $\PD_X$ means that we are considering the Poincaré duality on the compact manifold $X$. 
	Now, $\PD_{\torus}\left(\left[f^{-1}(j(M))\right]\right)=\PD_{\torus}\left(\left[\{(0,0)\}\right]\right)$ generates $\cohomol{2}{\torus}\simeq \integ$; in other words, $\PD \left(j_\ast \left[M\right]\right)$ is toroidal. 
	As $\faktor{\cohomolTwoV}{\cohomolVator}$ is torsion-free, $c_1(\xi)$ is also toroidal.
	
	The only thing left to show is that $c_1(\pi_k^*\xi)=kc_1(\xi)\bmod\cohomolVator$ for each $k\geq 2$.
	\\
	Because $\eta$ is a trivial complex vector bundle over $\torus\times M$, the same is true for each $\pi_k^*\eta$; in particular, each $\pi_k^*\eta$ has trivial Chern classes. 
	Notice that $\pi_k^*\xi$ can also be seen as obtained from $\pi_k^*\eta$ by performing a half Lutz-Mori twist along each of the $k$ submanifolds $\left\{\left(\frac{2l\pi}{k},0\right)\right\}\times M$, with $l=0,\ldots,k-1$. 
	Then, Proposition \ref{PropCompChernClass} tells that $c_1(\pi^*_k\xi) = - 2 k \PD\left(j_\ast \left[M\right]\right) = k c_1(\xi)$, so that $c_1(\pi^*_k\xi) = k c_1(\xi)\bmod\cohomolVator$ too.
	
	\section{Examples from adapted open books and the h-principle}
	\label{SecExFromHPrinciple}
	
	In this section, we show how to obtain examples of $(\cercle\times W,\xi)$ as in the hypothesis of Theorem \ref{ThmExNonTrivContMappClass} using the existence of adapted open book decompositions due to Giroux \cite{Gir02} and the h-principle of Borman, Eliashberg and Murphy \cite{BorEliMur15}. 
	
	In the following, we are going to adopt two (homotopically equivalent) points of view on (co-orientable) almost contact structures on $V^{2n+1}$.
	More precisely, in Sections \ref{SubSecProofPropOBDBourgConstr} and \ref{SubSecProofPropHPrinciple} we look at them as, respectively, couples $(\xi,\omega_\xi)$ and $(\xi,J_\xi)$, where $\xi$ is a co-orientable hyperplane field on $V$, $\omega_\xi$ is a symplectic structure on $\xi$ and $J_\xi$ is a complex structure on it.

	\subsection{Proof of Theorem \ref{ThmExamplesOBDBourgConstr}}
	\label{SubSecProofPropOBDBourgConstr}
	
	In order to prove Theorem \ref{ThmExamplesOBDBourgConstr}, we need the following lemma which describes the effects of the Bourgeois construction \cite{Bou02} and of its branched coverings at the level of almost contact structures as well as a sufficient condition for overtwistedness in the case of branched covers:
	\begin{lemma}
		\label{LemmaBourgContStr1}
		Let $(V^{2n-1},\eta)$ be a contact manifold, where $\eta$ is co-orientable, $(B,\varphi)$ an open book decomposition supporting $\eta$ and $\alpha$ a contact form defining $\eta$ and adapted to the open book.
		Then, we have the following:
		\begin{enumerate}
			\item \label{Item1LemmaBourgContStr1} The Bourgeois construction \cite{Bou02} on $(V,\eta)$ and $(B,\varphi,\alpha)$ gives a contact structure $\xi$ on $V\times\torus$ which is homotopic, as an almost contact structure, to $\left(\eta\oplus T\torus,d\alpha\oplus\omega_{T}\right)$, where $\omega_{T}$ is a volume form on $\torus$.
			\item \label{Item2LemmaBourgContStr1} Any contact branched covering $\xi_g$ of $\xi$ via a branched covering $\nu\co V\times\Sigma_g\rightarrow V\times\torus$, induced by a covering $\Sigma_g\rightarrow\torus$ branched over two points, is homotopic, as an almost contact structure, to $\left(\eta\oplus T\Sigma_g,d\alpha\oplus\omega_g\right)$, where $\omega_g$ is a volume form on $\Sigma_g$.
			\item \label{Item3LemmaBourgContStr1} Suppose $\eta$ is overtwisted. Then, if $g$ is large enough, $\xi_g$ is overtwisted too.
		\end{enumerate}
	\end{lemma}
	Notice that point \ref{Item1LemmaBourgContStr1} above has already been pointed out by Lisi, Marinkovi\'c and Niederkr\"uger \cite[Remark 2.1]{LisMarNie18}.
	\\
	We now prove, in this order, Theorem \ref{ThmExamplesOBDBourgConstr} and Lemma \ref{LemmaBourgContStr1}:
	
	\begin{proof}[Proof (Theorem \ref{ThmExamplesOBDBourgConstr})]
		We use the notations of Theorem \ref{ThmExamplesOBDBourgConstr}. 
		Denote also the natural projections by
		\begin{equation*}
		p\co V\times\torus\rightarrow V \text{ , } \;
		p_g\co V\times\Sigma_g\rightarrow V \, \text{ and } \;
		p_g'\co V\times\Sigma_g\rightarrow\Sigma_g \text{ .}
		\end{equation*}
		Points \ref{Item1LemmaBourgContStr1} and \ref{Item2LemmaBourgContStr1} of Lemma \ref{LemmaBourgContStr1} imply that $c_1(\xi)=p^*c_1(\eta)$ and $c_1(\xi_g)=p_g^*c_1(\eta) + (p_g')^*\chern_1(T\Sigma_g)$.
		Recall now that every continuous map from $\torus$ to $\Sigma_g$ has degree $0$ (here, we use $g\geq 2$); in particular, for each $f\co \torus \rightarrow V\times\Sigma_g$, we have
		\begin{equation*}
		f^*(p_g')^*c_1(T\Sigma_g)=(p_g'\circ f)^* c_1(T\Sigma_g)=0 \in \cohomol{2}{\torus} \text{ ,}
		\end{equation*}
		i.e. $(p_g')^*\chern_1(T\Sigma_g)$ is atoroidal.
		We hence have that 
		\begin{equation}
		\label{EqProofThm3}
		\begin{split}
		&c_1(\xi)=p^*c_1(\eta)\bmod \cohomolVtorusator \, , \\
		&c_1(\xi_g)=p_g^*c_1(\eta)\bmod \cohomolVSigmagator \text{ .}
		\end{split}
		\end{equation}
		
		We now claim that both $p$ and $p_g$ pull-back toroidal classes on $V$ to toroidal classes on, respectively, $V\times\torus$ and $V\times\Sigma_g$. By Equation \ref{EqProofThm3} and the fact that $c_1(\eta)$ is toroidal by hypothesis, this would then directly imply that $c_1(\xi)$ and $c_1(\xi_g)$ are toroidal too. 
		\\
		Let $a\in\cohomolV{2}$ be toroidal, i.e. there is $t\co \torus\rightarrow V$ with $t^*a\neq 0$; we then want to prove that $p^*a\in \cohomol{2}{V\times\torus}$ is toroidal too. 
		Consider any $h\co \torus \rightarrow V\times\torus$ such that $p\circ h = t$; for instance, let $q_0\in\torus$ and take $h(.)\coloneqq(t(.),q_0)$. Then,
		\begin{equation*}
		h^*(p^*a)=(p\circ h)^*a=t^*a\neq 0 \in \cohomol{2}{\torus} \text{ ,}
		\end{equation*}
		i.e. $p^*a$ is toroidal, as desired.
		An analogous argument shows that $p_g^*a$ is toroidal too.
		
		The fact that $\xi$ and $\xi_g$ satisfy
		\begin{align*}
		&c_1(\mu_k^*\xi)=kc_1(\xi)\bmod \cohomolVtorusator \text{ ,}\\
		&c_1((\mu_k^g)^*\xi_g)=kc_1(\xi_g)\bmod \cohomolVSigmagator
		\end{align*}
		follows, by a direct computation, from Equation \ref{EqProofThm3}, from the equalities $\pi_k\circ p = p\circ \mu_k$, $\pi_k\circ p_g = p_g\circ \mu_k^g$ and from the fact that $c_1(\pi_k^*\eta)=kc_1(\eta)\bmod \cohomolVator$.
		
		Lastly, if $\eta$ is overtwisted, point \ref{Item3LemmaBourgContStr1} of Lemma \ref{LemmaBourgContStr1} gives the overtwistedness of $\xi_g$ for $g$ large enough, thus concluding the proof. 
	\end{proof}
	
	\begin{proof}[Proof (Lemma \ref{LemmaBourgContStr1})]
		We start by proving point \ref{Item1LemmaBourgContStr1}.
		The Bourgeois construction \cite{Bou02} on $(V,\eta)$ and $(B,\varphi,\alpha)$ gives a function $\Phi=(f,g)\co V\rightarrow \real^2$ defining the open book $(B,\varphi)$ and such that $\xi$ on $V\times\torus_{(x,y)}$ is defined by $\beta\coloneqq \alpha + f dx - gdy$.
		Then, an explicit homotopy of almost contact structures from $(\xi,d\beta\vert_{\xi})$ to $(\eta\oplus T\torus,d\alpha\vert_{\eta}+dx\wedge dy)$ is given by the $[0,1]_t$-family of hyperplane fields $\xi_t$ given by the kernel of $\alpha + \left(1-t\right)\left(fdx-gdy\right)$, together with the symplectic structures given by the restriction of $d\alpha + (1-t) \left[df\wedge dx - dg\wedge dy\right] + tdx\wedge dy$ to $\xi_t$.
		
		As far as point \ref{Item2LemmaBourgContStr1} is concerned, as explained in Geiges \cite{Gei97}, an explicit contact branched covering $\xi_g$ on $V\times\Sigma_g$ is given by the kernel of a differential $1-$form $\nu^*\beta + \epsilon h(r)r^2d\theta$;
		here, $(r,\theta)$ are radial coordinates on 
		the $D^2$-factor of a neighborhood $D^2\times\{p,q\}$ of 
		the upstairs branching locus $\{p,q\}$ of the branched covering $\Sigma_g\rightarrow\torus$, $\epsilon>0$ is very small and $h=h(r)$ is a smooth function with support in
		%$D^2\sqcup D^2$, 
		$D^2\times\{p,q\}$,
		equal to $1$ on the branching locus and strictly decreasing in $r$.
		As contact branched coverings are unique up to isotopy (see Gironella \cite[Section 2.2]{MyPaper17b}), it's enough to prove that this specific $\eta_g$ is homotopic to the desired almost contact structure.
		\\
		Now, an explicit computation (analogous to the one in \cite[Section 6.5]{MyPaper17b}) shows that the desired homotopy of almost contact structures is given by the $[0,1]_t$-family of hyperplane fields $\xi_g^t$ defined as the kernel of $\nu^*\alpha + \left(1-t\right)\left[\nu^*\left(fdx - gdy\right) + \epsilon h r^2 d\theta\right]$, together with the symplectic structures given by the restriction of $\nu^*d\alpha+\left(1-t\right)\left[\nu^*\left(df\wedge dx - dg\wedge dy\right) + \epsilon d\left(h r^2\right)\wedge d\theta\right] + t \omega_g$ to $\xi_g^t$.
		
		Point \ref{Item3LemmaBourgContStr1} has already been discussed in \cite[Section 7.2]{MyPaper17b}; more precisely, it essentially follows from the following three facts. 
		Firstly, the contact branched covering $\xi_g$ can be chosen (up to isotopy) in such a way that it induces on each fiber of $V\times\Sigma_g\rightarrow\Sigma_g$ the original overtwisted contact structure $\eta$. 
		Secondly, Niederkr\"uger and Presas \cite[page 724]{NiePre10} describe how the ``size'' of a contact neighborhood of each connected component  $(V,\xi)$ of the branching set of $V\times\Sigma_g\rightarrow V\times\torus$ 
		is diverging to $+\infty$ as the index $g$ of the branched covering is going to $+\infty$; see also \cite[Lemma 7.10]{MyPaper17b}.
		Then, according to Casals, Murphy and Presas \cite[Theorem 3.1]{CMP15}, topologically trivial contact neighborhoods of overtwisted manifolds in codimension $2$ are themselves overtwisted, provided they are sufficiently ``large''. This concludes the proof of Lemma \ref{LemmaBourgContStr1}.
	\end{proof}

	\subsection{Proof of Proposition \ref{PropExamplesHPrinciple}}
	\label{SubSecProofPropHPrinciple}
	
	The proof is structured as follows. 
	We start from a natural almost contact structure $\eta_0$ on $V\coloneqq \cercle \times W$ and we modify it to an almost contact structure $\eta$ with first Chern class $c_1(\eta)$ satisfying the desired conditions.
	Then, the h-principle from Borman, Eliashberg and Murphy \cite{BorEliMur15} tells that $\eta$ can be deformed to an overtwisted contact structure $\xi$ on $V$; the first Chern class of such a $\xi$ will then satisfy the desired properties too.
	
	Before entering in the details of the proof of Proposition \ref{PropExamplesHPrinciple}, we state a lemma from algebraic topology, whose proof is postponed:
	\begin{lemma}
		\label{LemmaAlmContStr1}
		Let $\eta_0$ be a (coorientable) almost contact structure on $V^{2n+1}$.
		For each $u\in\cohomol{2}{V}$, there is an almost contact structure $\eta_u$ on $V$ with $c_1(\eta_u)= c_1(\eta_0)+2u$.
	\end{lemma}

	\begin{proof}[Proof (Proposition \ref{PropExamplesHPrinciple})]
		The hyperplane field $\eta_0=\{0\}\oplus TW$ on $V=\cercle\times W$ is a (coorientable) almost contact structure thanks to the almost complex structure $J_W$ on $W$.
		Moreover, its first Chern class $c_1(\eta_0)$ is equal to $\pi_W^*c_1(W)$, where $\pi_W\co \cercle\times W \rightarrow W$ is the projection on the second factor.
		\\
		The hypothesis that $W$ is spin means that the $2$nd Stiefel Whitney class $w_2(W)\in H^{2}(W;\integ_{2})$ of $W$ is trivial. 
		Because $w_2(W)$ is the reduction modulo $2$ of $c_1(W)$, there is $\lambda \in \cohomol{2}{W}$ such that $c_1(W)=2\lambda$. 
		Hence, $c_1(\eta_0)=\pi_W^*c_1(W)=2\pi_W^*\lambda$.
		
		Consider then a non-trivial $c\in\cohomol{1}{W}\neq\{0\}$, %that exists by hypothesis,
		and let $v$ be a generator of $\cohomol{1}{\cercle}$. 
		Using Kunneth's decomposition theorem, we can see $\cohomol{1}{\cercle}\otimes\cohomol{1}{W}$ as a submodule of $\cohomol{2}{\cercle\times W}$.
		An application of Lemma \ref{LemmaAlmContStr1} with $u=v\otimes c - \pi_W^*\lambda$ then gives an almost contact structure $\eta$  with $c_1(\eta)= 2 v\otimes c$.
		
		Notice that the map $\pi_k^*$, induced on $\cohomol{2}{\cercle\times W}$ by $\pi_k$, acts as multiplication by $k$ on the submodule $\cohomol{1}{\cercle}\otimes\cohomol{1}{W}$ of $\cohomol{2}{\cercle\times W}$.
		In particular, the fact that $c_1(\eta)= 2 v\otimes c$ implies that $c_1(\pi_k^*\eta) = k c_1(\eta)\bmod\cohomolVator$.

		We also claim that $c_1(\eta)$ is toroidal.
		Indeed, according to the universal coefficient theorem and the Hurewicz theorem, $\cohomol{1}{W}\simeq \Hom_{\integ}\left(\homol{1}{W};\integ\right)\simeq \Hom_{\integ}\left(\pi_1(W);\integ\right)$; 
		in particular, as $c\neq 0 \in \cohomol{1}{W}$, there is $\gamma\co \cercle\rightarrow W$ such that $\gamma^*c\neq 0 \in \cohomol{1}{\cercle}$.
		If we define $f\coloneqq(\Id,\gamma)\co \torus=\cercle\times\cercle\rightarrow\cercle\times W$, we then have $f^{*}c_1(\eta)=2 v \otimes \gamma^*c\neq 0$ in $\cohomol{1}{\cercle}\otimes\cohomol{1}{\cercle}\subset \cohomol{2}{\torus}$, i.e. $c_1(\eta)$ is toroidal, as desired.
		
		The h-principle from Borman, Eliashberg and Murphy \cite{BorEliMur15} then gives the desired contact structure $\xi$ as deformation of $\eta$. 
	\end{proof}

	We now give a proof of the lemma used above:
	\begin{proof}[Proof (Lemma \ref{LemmaAlmContStr1})]
		Bowden, Crowley and Stipsicz \cite[Lemma 2.17.(1)]{BowCroSti14} states that if $V$ is a closed connected manifold of dimension $2n+1$ and $\zeta$ is a stable almost complex structure on it, then there is an almost contact structure $\eta$ on $V$ whose stabilization gives $\zeta$. 
		Recall that a \emph{stable almost complex structure} on $V$ is the stable isomorphism class of a complex structure on $TV\oplus \varepsilon_V^k$, where $\varepsilon_V$ is the trivial real vector bundle of dimension $1$ over $V$,
		and the \emph{stabilization of $\eta$} is the stable isomorphism class of the  complex structure induced by $\eta$ on $TV\oplus \varepsilon_V$.
		In particular, in order to prove Lemma \ref{LemmaAlmContStr1}, it's enough to find a stable almost complex structure $\zeta_u$ such that $c_1(\zeta_u)=c_1(\eta_0)+2u$. 
		
		The existence of such a $\zeta_u$ follows, for instance, from Geiges \cite[Remark 8.1.4]{Gei08}, of which we recall here the idea.
		\\
		There is a bijective correspondence, given by the first Chern class, between isomorphism classes of complex line bundles over $V$ and cohomology classes in $\cohomol{2}{V}$. 
		Let then $L_u$ be the complex line bundle over $V$ satisfying $c_1(L_u)=u$. 
		Consider then a complex vector bundle $E_u$ over $V$ such that there are $m\in \nat_{>0}$ and an isomorphism $\nu\colon L^*_u\oplus_\compl E_u\simeq (\varepsilon_V^\compl)^m$ of complex vector bundles over $V$, where $\varepsilon_V^\compl$ denotes the  
		complexification of $\varepsilon_V$; for a proof of the existence of such a complement $E_u$ see for instance Atiyah \cite[Corollary 1.4.14]{Ati89}.
		We then claim that the complex vector bundle $F_u\coloneqq\eta_0\oplus L_u \oplus E_u$ can be used to define the desired stable complex structure.
		\\
		The fact that $L^*_u\oplus_\compl E_u$ is a trivial complex vector bundle implies in particular that  $c_1(E_u)=-c_1(L^*_u)=u$; hence, $c_1(F_u)=c_1(\eta)+u+u = c_1(\eta)+2u$.\\
		Now, because $L^*_u$ and $L_u$ are isomorphic as real vector bundles, $\nu$ induces an isomorphism of real vector bundles $\nu'\colon L_u\oplus_{\real} E_u\simeq \varepsilon_V^{2m}$. 
		Moreover, the choice of a vector field $X$ on $V$ transverse to $\eta_0$  
		gives an isomorphism of real vector bundles $\Psi\colon \eta_0 \oplus\varepsilon_V\simeq TV$.
		We then have an isomorphism $\theta$ of real vector bundles over $V$ given by the composition
		\begin{equation*}
		F_u = \eta_0\oplus L_u \oplus E_u \overset{\Id\oplus\nu'}{\simeq} \eta_0 \oplus \varepsilon^{2m}_V=_\real \left(\eta_0 \oplus \varepsilon_V\right)\oplus\varepsilon^{2m-1}_V \overset{\Psi\oplus\Id}{\simeq} TV \oplus\varepsilon^{2m-1}_V \text{ .}
		\end{equation*} 
		In particular, the pushforward $\theta_* J$ of the complex structure $J$ on $F_u$ via $\theta$ gives the desired stable almost complex structure $\zeta_u$ on $V$. 
	\end{proof}

	\bibliographystyle{gtart}

	\Addresses

\end{document}